\documentclass[12pt, reqno]{amsart}
\usepackage{amssymb, amsthm, amsmath, amsfonts, bbm}
\usepackage[shortlabels]{enumitem}
\usepackage{array}	
\usepackage{lineno}

\setlist*{label=\alph*), ref=\alph*, wide, labelindent =0em, labelwidth=!}

\usepackage{color}

\usepackage[hidelinks]{hyperref}

\usepackage[makeroom]{cancel} 

\newcommand*\patchAmsMathEnvironmentForLineno[1]{%
  \expandafter\let\csname old#1\expandafter\endcsname\csname #1\endcsname
  \expandafter\let\csname oldend#1\expandafter\endcsname\csname end#1\endcsname
  \renewenvironment{#1}%
     {\linenomath\csname old#1\endcsname}%
     {\csname oldend#1\endcsname\endlinenomath}}%
\newcommand*\patchBothAmsMathEnvironmentsForLineno[1]{%
  \patchAmsMathEnvironmentForLineno{#1}%
  \patchAmsMathEnvironmentForLineno{#1*}}%
\AtBeginDocument{%
\patchBothAmsMathEnvironmentsForLineno{equation}%
\patchBothAmsMathEnvironmentsForLineno{align}%
\patchBothAmsMathEnvironmentsForLineno{flalign}%
\patchBothAmsMathEnvironmentsForLineno{alignat}%
\patchBothAmsMathEnvironmentsForLineno{gather}%
\patchBothAmsMathEnvironmentsForLineno{multline}%
}

  \evensidemargin
\oddsidemargin
\newcommand{\supp}{\mathop{\mathrm{supp}}\nolimits}

\makeatletter
\DeclareFontFamily{U}{tipa}{}
\DeclareFontShape{U}{tipa}{m}{n}{<->tipa10}{}
\newcommand{\arc@char}{{\usefont{U}{tipa}{m}{n}\symbol{62}}}%

\newcommand{\arc}[1]{\mathpalette\arc@arc{#1}}

\newcommand{\arc@arc}[2]{%
  \sbox0{$\m@th#1#2$}%
  \vbox{
    \hbox{\resizebox{\wd0}{\height}{\arc@char}}
    \nointerlineskip
    \box0
  }%
}
\makeatother

\DeclareMathOperator{\conv}{conv}

\DeclareMathOperator{\cl}{cl}

\DeclareMathOperator{\aff}{aff}
\DeclareMathOperator{\rint}{rint}
\DeclareMathOperator{\intr}{int}

\DeclareMathOperator{\pr}{Pr}

\newtheorem*{def*}{Definition}

\newtheorem*{rem*}{Remark}
\newtheorem*{cor*}{Corollary}
\newtheorem{cor}{Corollary}

\newtheorem{prop}{Proposition}
\newtheorem{lem}{Lemma}
\newtheorem*{lem1'}{Lemma $\mathbf{1^\prime}$}

\newtheorem{theorem}{Theorem}
\newtheorem{lemma}{Lemma}

\theoremstyle{definition}
\newtheorem{definition}{Definition}

\theoremstyle{remark}
\newtheorem{remark}{Remark}

\newtheorem{example}{Example}

\def\R{\mathbb{R}}

\def\N{\mathbb{N}}
\def\E{\mathbb{E}}
\def\P{\mathbb{P}}

\def\S{\mathbb{S}}
\def\I{\mathbbm{1}}

\newcommand{\bt}{\begin{theo}}
\newcommand{\et}{\end{theo}}
\newcommand{\bl}{\begin{lem}}
\newcommand{\el}{\end{lem}}
\newcommand{\bc}{\begin{cor*}}
\newcommand{\ec}{\end{cor*}}
\newcommand{\br}{\begin{rem*}}
\newcommand{\er}{\end{rem*}}
\newcommand{\bp}{\begin{proof}}
\newcommand{\ep}{\end{proof}}
\newcommand{\bes}{\begin{ex}}
\newcommand{\ees}{\end{ex}}

\begin{document}

\title[{When is the rate function of a random vector strictly convex?}]{When is the rate function \\of a random vector strictly convex?}

\author{Vladislav~Vysotsky}
\address{Vladislav~Vysotsky, University of Sussex}
\email{v.vysotskiy@sussex.ac.uk}

\begin{abstract}
We give a necessary and sufficient condition for strict convexity of  the rate function of a random vector  in $\R^d$. This condition is always satisfied when the random vector has finite Laplace transform. We also completely describe the effective domain of the rate function under a weaker condition.
\end{abstract}


\maketitle

\section{Introduction}

Let $X$ be a random vector  in $\R^d$ and $I_X$ be its {\it rate function}, given by
\[
I_X(v):= \sup_{u \in \R^d} \bigl( u \cdot v - \log \E e^{u \cdot X} \bigr), \qquad v \in \R^d,
\] 
where `$\cdot $' stands for the scalar product in $\R^d$. This function is the {\it convex conjugate} of the {\it logarithmic Laplace transform} of $X$, defined by $K_X(u):=\log \E e^{u \cdot X}$ for every $u \in \R^d$.

The function $K_X$ takes values in $(-\infty, +\infty]$, satisfies $K(0)=0$, and is convex by H\"older's inequality.  Then $I_X$ is also convex and finite at least at one point (\cite[Theorem~12.2]{Rockafellar}), and it takes values in $[0, +\infty]$. The {\it effective domain} $\mathcal D(I_X)$ of $I_X$, defined~by 
\[
\mathcal D(I_X):=\{v\in \R^d: I_X(v)<+\infty\},
\] 
is convex and non-empty, and so is the effective domain $\mathcal D(K_X)$ of $K_X$.


It is easy to show that $K_X$ is differentiable at every point of $\intr \mathcal D(K_X)$ (\cite[\mbox{Corollary}~7.1]{B-N}). When the set $\intr \mathcal D(K_X)$ is non-empty, we say that $K_X$ is {\it steep} (at the boundary of its effective domain) if $\lim_{n\to \infty}|\nabla K_X(u_n)| = \infty$ for every sequence $u_1, u_2, \ldots $ in $\intr \mathcal D(K_X)$ converging to a point in  $\partial \mathcal D(K_X)$. 
Note that $K_X$ is steep  when it is finite at every point.

The property of steepness appears in a number of general convex-analytic results concerning the so-called {\it essentially smooth} convex functions on $\R^d$ (\cite[Section~26]{Rockafellar}). In the  context of large deviations probabilities this property features in the important G\"artner--Ellis theorem (\cite[Section~2.3]{DemboZeitouni}). The assumption $0 \in \intr \mathcal D(K_X)$, which is of course stronger than $\intr \mathcal D(K_X) \neq \varnothing$, is crucial for classical Cram\'er's theorem  (\cite[Section~2.4]{FirasTimo}), where the rate function $I_X$ plays the key role. 

Let us recall a few more definitions. For any $A \subset \R^d$, denote by $\conv A$ (resp.\ $\aff A$) the {\it convex hull} (resp.\ {\it affine hull}) of $A$, i.e.\ the minimal convex (resp. affine) subset of $\R^d$ containing $A$; denote by $\rint A$ is the {\it relative interior} of $A$, i.e.\ the interior of $A$ in the relative topology of $\aff A$; and let $\partial_{\text{rel}} A:= \cl A \setminus \rint A$ be the {\it relative boundary} of $A$. Note that $\rint A = A$ if $A$ consists of a single point. 

The {\it topological support} of a random vector $X$ in $\R^d$, denoted by  $\supp X$, is the minimal by inclusion closed set $S \subset \R^d$ such that $\P(X \in S)=1$. The {\it convex support} of~$X$ is
\[
C_X:=\conv( \supp X ).
\]

A hyperplane $L \subset \R^d$ {\it supports} a convex set $C \subset \R^d$ if $L$ intersects $C$ and $C$ is contained in either of the two half-spaces of $\R^d$ that have $L$ as their boundary ($C \subset L$ is possible).

We say that $I_X$ is  {\it strictly convex on a set} $A \subset \mathcal D(I_X)$ if $I_X$ is affine on no line segment contained in $A$, and $I_X$ is {\it strictly convex} if it is strictly convex on $\mathcal D(I_X)$.

Our starting point is the following assertion.

\begin{prop} \label{prop: known}
Let $X$ be a random vector in $\R^d$, $d \ge 1$. 
\begin{enumerate}
\item \label{item: D_I inclusions} 
We have
\begin{equation} \label{eq: D_I inclusions}
\rint C_X \subset \mathcal{D} (I_X)  \subset  \cl C_X, 
\end{equation}
hence $\rint \mathcal{D} (I_X) = \rint C_X$. Moreover, $I_X=+\infty$ on
every hyperplane $L$ in $\R^d$ supporting $\cl C_X$ and such that $\P(X \in L)=0$.
\item \label{item: I strictly convex}
$I_X$ is strictly convex on $\rint \mathcal{D} (I_X)$ if and only if $\intr \mathcal D(K_X) \neq \varnothing$ and $K_X$ is steep.
\end{enumerate}
\end{prop}
We will prove this result  in full for completeness of exposition.
The inclusions in~\eqref{eq: D_I inclusions} are not new but we do not have exact references. They are stated in \cite[Theorem~9.1]{B-N}, which  however concerns only a specific type of distributions. They  follow from \cite[Theorems 2.1, 2.3,~3.2]{BahadurZabell} but do not appear in \cite{BahadurZabell} explicitly. The last claim of Part~\ref{item: D_I inclusions}) is in \cite[Theorem~9.5]{B-N}.  
Part~\ref{item: I strictly convex}) states a particular case of a general convex-analytic result~\cite[Theorem~26.3]{Rockafellar}, with the novelty that we strengthened the direct implication. 

The main result of this note is a necessary and sufficient condition (see Theorem~\ref{thm: main}) for  strict convexity of $I_X$ on its whole effective domain $\mathcal D(I_X)$. This condition is always satisfied when the Laplace transform of $X$ is finite in the whole of $\R^d$, and thus $I_X$ is strictly convex for such $X$. 

In view of Proposition~\ref{prop: known}, we only need to characterize strict convexity of the rate function on the relative boundary of $\mathcal D(I_X)$. 
Our approach is based on the following result. 
\begin{theorem} \label{thm: restriction}
Let $X$ be a random vector in $\R^d$, $d \ge 1$, and $L$ be a hyperplane in $\R^d$ supporting $C_X$ and such that $0< \P(X \in L) <1 $. Then
\begin{equation} \label{eq: I_X = I_X|L}
I_X(v)=  I_{X|L}(v) - \log \P(X \in L), \qquad v \in L,
\end{equation}
if and only if
\begin{equation} \label{eq: consistent proj}
\pr_L  (\rint \mathcal D(K_{X}) ) =  \pr_L  (\rint \mathcal D(K_{X|L}) ), 
\end{equation}
where $X|L$ is a random vector distributed as $X$ conditioned to be in $L$ and $\pr_L$ denotes the orthogonal projection from $\R^d$ onto $L$.
\end{theorem}

Assume that condition \eqref{eq: consistent proj}  is satisfied for every hyperplane $L$ as above. Then we can apply Proposition~\ref{prop: known} to each of the $I_{X|L}$'s. By \eqref{eq: I_X = I_X|L}, this ensures that $\rint \mathcal D(C_{X|L}) \subset \mathcal D(I_X)$, and  $I_X$ is strictly convex on every set $\rint \mathcal D(C_{X|L})$ if and only if the $K_{X|L}$'s are steep. The main idea is to apply this argument to the $X|L$'s and further on {\it recursively}, using that the sets $\rint \mathcal{D} (I_{X|L})$ are disjoint with $\rint \mathcal{D} (I_X) $  by $\dim C_{X|L}<\dim C_X$. Under appropriate conditions, which ensure that \eqref{eq: consistent proj}  is satisfied  at every step of the recursion, this lets us fully describe $\mathcal D(I_X)$ (see Corollary~\ref{cor: eff domain}) and characterize strict convexity of $I_X$ (see Theorem~\ref{thm: main}). We give the details in the next section, where we also comment on condition \eqref{eq: consistent proj} (see Remark~\ref{rem: Cond a}).

The property of strict convexity can be useful when proving uniqueness of solutions to minimization problems involving $I_X$. Such problems arise from large deviations principles, most naturally in Cram\'er's theorem  (see \cite[Section~2.4]{FirasTimo})  on random walks in $\R^d$. There are functional versions of this result, which describe scaled trajectories of random walks and continuous time analogues for L\'evy processes~(see \cite[Section~5.1 and~5.2]{DemboZeitouni} and \cite{Mogul}). In both cases, if the increments have finite Laplace transform, then the large deviations are described by the rate function $I$ of the form $I(f)=\int_0^1 I_X(f'(t))dt$ for $f$ in $AC_0$, the space of coordinate-wise absolutely continuous $\R^d$-valued functions on $[0,1]$ such that $f(0)=0$.

For concrete examples, let $(S_n)_{n \in \N}$ be a random walk with i.i.d.\ increments distributed as $X$.  When $I_X$ is strictly convex, its unique minimizer $b$  over a  convex Borel set $B \subset \R^d$ that meets $\mathcal D(I_X)$ can be interpreted as the limit constant in the law of large numbers for the averages $S_n/n$ conditioned to be in $B$. Under this conditioning, a typical trajectory $(S_k/n)_{1 \le k \le n}$ of the random walk with $\mathcal D(K_X) =\R^d$ is asymptotically linear with slope $b$ because the function $f_0(t)=bt$ is the unique minimizer of $I$ over the set $\{f \in AC_0: f(1) \in B \}$. This  follows from Jensen's inequality using that $I_X$ is strictly convex  (by Corollary~\ref{cor: strictly convex}). When $0 \in \intr \mathcal D(K_X)$ but $\mathcal D(K_X) \neq \R^d$, the rate function $I$ has a more complicated form, and without strict convexity~of~$I_X$ the argument above becomes less simple  (see~\cite[pp.~16-17]{Lifshits}).
More elaborate examples arise, e.g.\ in the study~\cite{AkopyanVysotsky} of large deviations of the perimeter and the area of convex hulls of planar random walks, where strict convexity of $I_X$ simplified~considerations. 

Finally, we note that relating the conditional limit laws to the minimizers of the rate function, as above, corresponds to the fundamental {\it Gibbs conditioning principle} of  statistical mechanics (see~\cite{DemboZeitouni}, including Sections~3.3 and~7.3).


\section{Main result}

We first recall some facts on the structure of convex sets.

A {\it face} of a non-empty convex set $C \subset \R^d$ is a convex subset $C'$ of $C$ such that every closed line segment in $C$ with a relative interior point in $C'$ has both endpoints in $C'$. Note that $C$ itself is a face; the zero-dimensional faces are called the {\it extreme points} of $C$. If $L \subset \R^d$ is a hyperplane supporting $C$, then $C \cap L$ is face of $C$. Every face of such form is called {\it exposed}.
 
Denote by $\mathcal F(C)$ the set of non-empty faces of $C$ and by $\mathcal F^*(C)$ its subset of {\it maximal proper faces}, defined by
\[
\mathcal F^*(C):= \big \{C' \in \mathcal F(C) \setminus \{C\} : C' \not \subset C'' \text{ for every } C'' \in \mathcal F(C) \setminus \{C, C'\} \big\}.
\] 
We will use extensively that every face in $\mathcal F^*(C)$ is exposed (this follows from \cite[Theorem~11.6 and Corollary~18.1.3]{Rockafellar}). Our need in the set $\mathcal F^*(C)$ is due to the following result.
\begin{lemma} \label{eq: C boundary}
Let $C \subset \R^n$ be a non-empty convex set. Then
\begin{equation} \label{eq: C without rint}
C \setminus \rint C= \bigcup_{C' \in \mathcal F^*(C)} C'.
\end{equation}
\end{lemma}
\begin{proof}
In fact, the set $C \setminus \rint C$ contains every proper face of $C$ by \cite[Corollary~18.1.3]{Rockafellar}. On the other hand, by  \cite[Theorem~11.6]{Rockafellar}, for every point in $C \setminus \rint C$  there is a hyperplane $L$ containing this point and supporting $C$ but not containing $C$. Then $C \cap L$ is a proper face of $C$. To finish the proof it remains to argue that for any proper face $C'$ of $C$ is contained in a maximal proper face of $C$. 

Let us use induction in $\dim C'$. In the base case $\dim C'=\dim C-1$, we always have $C' \in \mathcal F^*(C)$. Indeed, if this were not true, there would be a proper face $C''$ of $C$ other than $C'$ that strictly contains $C'$. Then $C'$ would be a face of $C''$ (by definition of a face), hence $\dim C' < \dim C''$ by \cite[Corollary~18.1.3]{Rockafellar}. This is a contradiction because there are no faces of $C$ other than itself of dimension $\dim C$. 

Let us prove the inductive step. If $C'$ is maximal, we are done. Otherwise, choose $C''$ as above. By the assumption of induction, there is a  $C''' \in \mathcal F^*(C)$ containing $C''$. This is a face required.
 \end{proof}

\medskip

We now consider faces of the convex support $C_X$ of a random vector $X$ in $\R^d$. First note that $C_X$ is not necessarily closed; it can be even open.
\begin{example}
Let $X$ be a random vector in $\R^2$ such that $\supp X=\{(x,y) \in \R^2: y \ge \frac{1}{1+x^2}\}$. Then $C_X$ is the open upper half-plane.
\end{example}
However, we have the following measurability result.
\begin{lemma} \label{lem: measurable}
Let $X$ be a random vector in $\R^d$,  $d \ge 1$. Then $C_X$ is a Borel subset of $\R^d$, and so is every $C \in \mathcal F^*(C_X)$.
\end{lemma}

\begin{proof}
By Carath\'eodory's theorem (\cite[Theorem~17.1]{Rockafellar}), every point in $C_X$ is a convex combination of $d+1$ points in $\supp X$. Then
$C_X = \cup_{n =1}^\infty \conv \big( \supp X \cap \{u \in \R^d: \| u \| \le n\} \big)$. By \cite[Theorem~17.2]{Rockafellar}, each set under the union is closed, and hence $C_X$ is Borel. 

Every $C \in \mathcal F^*(C_X)$ is an exposed face of $C_X$, therefore $C=C_X \cap L$ for some affine hyperplane $L$ supporting $C_X$. Hence $C$ also is a Borel set.
\end{proof}



The lemma ensures that the following set is well-defined:
\[
\mathcal F^*_+(C_X):=\{C \in \mathcal F^*(C_X): \P(X \in C) >0\}.
\]
In the results below it is useful to know when this set is empty. We give the following~criterion.
\begin{lemma} \label{lem: F^*_+ empty}
Let $X$ be a random vector in $\R^d$,  $d \ge 1$. Then $\mathcal F^*_+(C_X) $ is empty if and only if there is no hyperplane $L$ in $\R^d$ supporting $C_X$ and such that $0<\P(X \in L)<1$.
\end{lemma}
\begin{proof}
If $L$ is a hyperplane supporting $C_X$, then either $C_X \subset L$, in which case $\P(X \in L)=1$, or $C_X \cap L \in \mathcal F^*(C_X)$, hence from $\mathcal F^*_+(C_X) = \varnothing$ we get $\P(X \in C_X \cap L)=0$ and thus $\P(X \in L)=0$. This proves the direct implication.

To prove the reverse implication, assume that there is a $C \in \mathcal F^*_+(C_X) $. This is an exposed face of $C_X$, therefore $C=C_X \cap L$ for some affine hyperplane $L$ supporting $C_X$. Then $\P(X \in L) = \P(X \in C)>0$, hence by the assumption, it must be $\P(X \in L)=1$. Hence $\supp X \subset L$ (because $L$ is a closed set) and therefore $C_X \subset L$. Thus, $C$ is not a proper face of $C_X$, which is a contradiction.
\end{proof}

For every random vector $X$ in $\R^d$ and $C \in \mathcal F^*_+(C_X)$, let $X|C$ be a random vector distributed as $X$ conditioned on $X \in C$.

We now give two key definitions, both having recursive structure.

\begin{definition} \label{def: proj property}
We say that $K_X$ has the {\it projection property}\footnote{Strictly speaking, this is a property of the distribution of $X$ rather than  of $K_X$. However, the distributions that satisfy $\intr \mathcal D(K_X) \neq \varnothing$  are determined by their Laplace transform (this reduces to $d=1$, where~Theorem~6a in Chapter~VI of~\cite{Widder} applies). Note that  our main result, Theorem~\ref{thm: main}.\ref{item: I strictly convex iff}, assumes $\intr \mathcal D(K_X) \neq \varnothing$.} if 
\begin{enumerate}
\item \label{item: consistent proj} for every hyperplane $L$ in $ \R^d$ supporting $C_X$ and such that $0<\P(X \in L)<1$, we have
\[
\pr_L  (\rint \mathcal D(K_{X}) ) = \pr_L  (\rint \mathcal D(K_{X|L}) ); 
\]
\item \label{item: K_X|C proj} $K_{X|C}$ has the projection property for every $C \in \mathcal F^*_+(C_X)$. 
\end{enumerate}
\end{definition}

The projection property is well-defined since the definition allows us to identify, using recursion in $\dim C_X$, whether each particular $K_X$ has this property nor not. This is true because 1)~$\dim C_{X|C} \le \dim C < \dim C_X$ for every $C \in \mathcal F^*_+(C_X)$; 2) the recursion terminates (confirming that $K_X$ has the property) if Conditions \ref{item: consistent proj}) and~\ref{item: K_X|C proj}) hold vacuously, namely when $\mathcal F^*_+(C_X) = \varnothing$ (by Lemma~\ref{lem: F^*_+ empty}); and 3) the recursion always terminates since $\mathcal F^*_+(C_X) = \varnothing$ when $\dim C_X=0$, i.e.\ $X$ is constant~a.s.  

\begin{remark} \label{rem: Cond a} 
Let us comment on Condition~\ref{item: consistent proj}). 
\begin{enumerate}
\item \label{item: proj equiv} 
Each set $ \mathcal D(K_{X|L})$ is a right cylinder. So is its relative interior, which satisfies
\begin{equation*} 
\pr_L  (\rint \mathcal D(K_{X|L}) ) = \rint (\pr_L  \mathcal D(K_{X|L}) ) = \rint (L \cap  \mathcal D(K_{X|L})) = L \cap  \rint \mathcal D(K_{X|L}),
\end{equation*}
where the first and the last equalities follow from~\cite[Theorem~6.6 and Corollary~6.5.1]{Rockafellar}.

\item \label{item: proj inclusion} We always have $\mathcal D(K_X) \subset \mathcal D(K_{X|L})$. This follows from
\begin{equation} \label{eq: K proj ineq}
K_X(u) \ge \log \E \big( e^{u \cdot X} \I_{\{X \in L\}} \big)  = K_{X|L}(u)  + \log \P(X \in L), \qquad u \in \R^d.
\end{equation}
Assume additionally that  $\mathcal D(K_X)$ is not entirely contained in the relative boundary of $ \mathcal D(K_{X|L})$; this holds, e.g.\ when  $\intr \mathcal D(K_X) \neq \varnothing$ or $0 \in \rint \mathcal D(K_{X|L}) $. Then  
\begin{equation} \label{eq: proj inclusion}
\pr_L  (\rint  \mathcal D(K_{X}) ) \subset  \pr_L  (\rint \mathcal D(K_{X|L}) )
\end{equation}
because $\rint \mathcal D(K_X) \subset \rint \mathcal D(K_{X|L})$ by \cite[Corollary 6.5.2]{Rockafellar}. Thus, \eqref{eq: consistent proj}  means that the projection of $\rint \mathcal D(K_{X})$ on $L$ does not increase if $X$ is replaced by $X|L$. 

\item \label{item: boundary directions} Every supporting hyperplane $L$ to $C_X$ is of the form $L=\{v \in \R^d: \ell \cdot v = h_{C_X}(\ell) \}$, where $\ell \in \S^{d-1}$ is a unit vector orthogonal to $L$ and $h_{C_X}$ is the {\it support function} of $C_X$ defined by $h_{C_X}(u):= \sup_{v \in C_X} u \cdot v$, $u \in \R^d$.
Since $\ell \in \mathcal D (h_{C_X})$ if and only if $C_X$ is bounded in direction~$\ell$ (equivalently, $\supp X$ is bounded in direction~$\ell$), this implies that $\{a \ell: a \ge 0\} \subset \mathcal D(K_X)$. It is therefore easy to see that $\pr_L (\rint \mathcal D(K_X)) = L$ when $\ell \in \intr \mathcal D (h_{C_X})$. Hence for such $\ell$ equality~\eqref{eq: consistent proj} always holds true by $\mathcal D(K_X) \subset \mathcal D(K_{X|L})$.

Thus, it suffices to check the assumption of Condition~a) only for hyperplanes supporting $C_X$ that are orthogonal to directions in the set $\partial \mathcal D (h_{C_X}) \cap \S^{d-1}$. For $d=2$ this set contains at most two directions because $\mathcal D (h_{C_X})$ is a convex cone.

\end{enumerate}
\end{remark}

We now give a few examples.

\begin{example} \label{ex: proj property}
$K_X$ has the projection property in the following cases:
\begin{enumerate}
\item $\mathcal F^*_+(C_X)$ is empty. In particular,  this holds true when $\P(X  \in \partial_{\text{rel}} C_X) = 0$; see~\eqref{eq: C without rint}. 

\item \label{item: ex finite Laplace} $\mathcal D(K_X) = \R^d$ or, equivalently, $\E e^{u \cdot X} < \infty$ for every $u \in \R^d$; cf.~\eqref{eq: K proj ineq}.

\item \label{item: 1D} $d=1$.

\item $d=2$ and  equality \eqref{eq: consistent proj} holds true for every line $L$ of the form $L=\aff C$, where $C\in \mathcal F^*_+(C_X)$ is  unbounded  (there are at most two such faces). \\
Indeed, such lines are orthogonal to the directions in $\partial \mathcal D (h_{C_X}) \cap \S^1$ and then Remark~\ref{rem: Cond a}.\ref{item: boundary directions} applies. Clearly, Condition~\ref{item: consistent proj}) in Definition~\ref{def: proj property} is satisfied by Example~\ref{ex: proj property}.\ref{item: 1D}   since $\dim C_{X|C} \le 1$ for every $C \in \mathcal F^*_+(C_X)$.
\end{enumerate}
\end{example}

Our second key definition is as follows. 

\begin{definition} \label{def: totally steep}
If $\mathcal D(K_X)$ has non-empty interior, we say that $K_X$ is {\it totally steep} if $K_X$ is steep and $K_{X|C}$ is totally steep for every $C \in \mathcal F^*_+(C_X)$. 
\end{definition} 

Again, this property is well-defined by recursion in $\dim C_X$ because  1)~$\dim C_{X|C} < \dim C_X$ for every $C \in \mathcal F^*_+(C_X)$; 2)  $\intr \mathcal D(K_{X|C}) \neq \varnothing$ for $C \in \mathcal F^*_+(C_X)$ by $ \mathcal D(K_X) \subset \mathcal D(K_{X|C})$ (cf.~\eqref{eq: K proj ineq}); 3)  $K_X$ is totally steep when it is steep and $\mathcal F^*_+(C_X) = \varnothing$; and 4) $K_X$ is totally steep when $\dim C_X=0$.

\begin{example} \label{ex: totally steep}
$K_X$ is totally steep if $\mathcal D(K_X) = \R^d$.
\end{example}

\begin{example} \label{ex: non totally steep}
Let us construct $K_X$ which neither has the projection property nor is totally steep. Put $X:=(\alpha X_1,\alpha X_2 + (1-\alpha) X_3)$, where $X_1$, $X_2$, $X_3$, $\alpha$ are independent non-negative random variables such that $X_1$ and $X_2$ have the standard exponential distribution with density $e^{-x}$ for $x>0$, $X_3$ has the absolutely continuous distribution with density proportional to $e^{-2x}/(1+x^3)$ for $x>0$, and $\P(\alpha=0)=\P(\alpha=1)=1/2$. 

We have $\mathcal D(K_{X_3})=(-\infty, 2]$ and it is easy to check that $K_{X_3}'(2-)<+\infty$, hence $K_{X_3}$ is not steep; and $K_{X_1}$ is steep.  Furthermore, $K_X(u_1, u_2)=\frac12 K_{X_1}(u_1) +  \frac12 K_{X_2}(u_2) + \frac12 K_{X_3}(u_2)$ for $u_1, u_2 \in \R$; the set $C_X$ is the closed positive quadrant in the plane; and $\mathcal F^*_+(C_X)=\{C\}$ with $C:=\{0\} \times [0, \infty)$. We can see that $K_X$ is steep but not totally steep because $\mathcal D(K_X)=(-\infty,1 ) \times (-\infty, 1)$ but  for the ordinate line $L=\aff C$ supporting $C_X$, the random vector $X|L$ is distributed as $(0, X_3)$ and thus $K_{X|L}$ is not steep. This also shows that Condition~\ref{item: consistent proj}) in Definition~\ref{def: proj property} is violated because $\pr_L (\rint \mathcal D(K_{X|L})) = \{0\} \times (-\infty, 2)$ but  $\pr_L (\rint \mathcal D(K_X)) = \{0\} \times (-\infty, 1)$, and thus $K$ does not have the projection property.
\end{example}

We are now ready to state the main result of the paper.

\begin{theorem} \label{thm: main}
Let $X$ be a random vector in $\R^d$,  $d \ge 1$. 
\begin{enumerate}
\item \label{item: I on faces} 
If $K_X$ satisfies Condition~\ref{item: consistent proj}) in Definition~\ref{def: proj property} of the projection property, then
\begin{equation} \label{eq: D I_X max faces}
\mathcal F^*(\mathcal D(I_X)) \subset \{\mathcal D(I_{X|C}) : C \in \mathcal F^*_+(C_X) \} \subset \mathcal F(\mathcal D(I_X)) \setminus \{\mathcal D(I_X) \}.
\end{equation} 

\item \label{item: I strictly convex iff}
$I_X$ is strictly convex if and only if  $\intr \mathcal D(K_X) \neq \varnothing$, $K_X$ has the projection property, and $K_X$ is totally steep.
\end{enumerate}
\end{theorem}

Let us present a few corollaries.

\begin{cor} \label{cor: strictly convex}
If $\E e^{u \cdot X} < \infty $ for every $u \in \R^d$, then $I_X$  is strictly convex.
\end{cor}

\begin{proof}
This follows directly from Part~\ref{item: I strictly convex iff}) using Examples~\ref{ex: proj property}.\ref{item: ex finite Laplace} and~\ref{ex: totally steep}.
\end{proof}

\begin{cor} \label{cor: eff domain}
If $K_X$ has the projection property, then $\mathcal D(I_X) \subset C_X $ and
\[
\mathcal D(I_X)= \rint C_X \cup \bigcup_{C_1 \in \mathcal F^*_+(C_X)} \rint C_{X|C_1} \cup \bigcup_{C_2 \in \mathcal F^*_+(C_{X|C_1})} \rint C_{X|C_2} \cup \ldots \cup \bigcup_{C_d \in \mathcal F^*_+(C_{X|C_{d-1}})} \rint C_{X|C_d}. 
\]
\end{cor}

\begin{proof}
We have
\[
\mathcal D(I_X)= \rint \mathcal D(I_X) \cup \bigcup_{C \in \mathcal F^*(\mathcal D(I_X))} C= \rint C_X \cup \bigcup_{C_1 \in \mathcal F_+^*(C_X)} \mathcal D(I_{X|C_1}),
\]
where the first equality follows from \eqref{eq: C without rint} and the second one follows from \eqref{eq: D I_X max faces}  and the fact that $\rint \mathcal D(I_X)=\rint C_X $ (see Proposition~\ref{prop: known}.\ref{item: D_I inclusions}). Then we establish the equality claimed  by  simple induction in $\dim C_X$  using that each  random vector $(X|C_1)|C_2$ has the same distribution as $X|C_2$. In the base case  $\dim C_X = 0$ the claim holds by $\mathcal D(I_X)= C_X= \rint  C_X$ and $\mathcal F^*_+(C_X)=\varnothing$. The same inductive argument establishes the inclusion $\mathcal D(I_X) \subset C_X $.
\end{proof}

\begin{cor} \label{cor: extremal D_I}
Assume that $K_X$ has the projection property. Then $v $ is an extreme point of $\mathcal D(I_X)$ if and only if $v$ is an extreme point of $C_X$ and $\P(X = v) >0$. For such $v$, we have $I_X(v)=- \log \P(X = v) $.
\end{cor}

\begin{proof}
Assume that $v$ is an extreme point of $C_X$ and $\P(X=v)>0$. We use induction in $\dim C_X$. In the base case  $\dim C_X = 0$, we simply have $I_X(v)=0= -\log(X=v)$. To prove the induction step for $\dim C_X \ge 1$, use that by \eqref{eq: C without rint} there is a face $C \in \mathcal F^*_+(C_X)$ that contains~$v$. Then $v$ is an extreme point of $C_{X|C}$ because $v \in C_{X|C}$ by $\P((X|C)=v)>0$ and $v$ is an extreme point of the convex set $C_X$ which contains $C_{X|C} $. 

Since $C$ an exposed face of $C_X$, there is a hyperplane $L$ supporting $C_X$ such that $C= C_X \cap L$. Then $X|C$ and $X|L$ have the same distribution since $\P(X \in L \setminus C) =0$, and by~\eqref{eq: I_X = I_X|L} and the assumption of induction we get
\[
I_X(v) = I_{X|C}(v)- \log \P(X \in C) = - \log \P((X|C) =v) - \log \P(X \in C) = -\log \P(X=v).
\]
Then $I_X(v)<\infty$, and thus $v \in \mathcal D(I_X)$. Hence $v$ is an extreme point of $\mathcal D(I_X)$ because $v $ is an extreme point of the convex set $C_X$ which contains $\mathcal D(I_X)$ by Corollary~\ref{cor: eff domain}.

Proving the reverse implication is similar. For the induction step, for $\dim \mathcal D(I_X) \ge 1$, use that by~\eqref{eq: C without rint}  there is a face $F \in \mathcal F^*(\mathcal D(I_X))$ that contains $v$. Then $v$ is an extreme point of $F$. By~\eqref{eq: D I_X max faces}, $F= \mathcal D(I_{X|C})$ for some face $C \in \mathcal F^*_+(C_X)$, and we can apply the assumption of induction as above. 
\end{proof}


\section{Proofs}

\begin{proof}[{\bf Proof of Proposition~\ref{prop: known}}] 
\ref{item: D_I inclusions}) Recall that $C_X=\conv (\supp X)$. Fix a $v \not \in \cl C_X$.  By \cite[Corollary~11.5.1]{Rockafellar},  there exists a non-zero $u_0 \in \R^d$ such that $u_0 \cdot x <u_0 \cdot v$ for any $x \in \supp X$. 
In other words, $u_0 \cdot X < u_0 \cdot v$ a.s. By the monotone convergence theorem, we get
\begin{equation} \label{eq: I = infty}
I(v)= \sup_{u \in \R^d} (u \cdot v - K_X(u)) \ge \sup_{a >0} (a u_0 \cdot v-K_X(a u_0))=-\inf_{a >0} (\log \E e^{ a(u_0 \cdot X - u_0 \cdot v) }) = +\infty.
\end{equation}
Thus, $\mathcal D(I_X) \subset \cl C_X$. 

Furthermore, if $L$ a hyperplane supporting $\cl C_X$, take any non-zero $u_0 \in \R^d$ orthogonal to $L$ and directed such that  $u_0 \cdot x \le u_0 \cdot v$ for any $x \in \supp X$ and $v \in L$,  that is $u_0 \cdot X \le u_0 \cdot v$ a.s. This inequality is strict if $\P(X \in L)=0$, in which case $I(v)=+\infty$ holds true by \eqref{eq: I = infty}, as required.

We now show that $\rint C_X \subset \mathcal D(I_X)$. Assume that this does not hold. Then, since $\mathcal D(I_X) \subset \cl C_X$ and the sets $C_X$ and $\mathcal D(I_X)$ are convex, we have $\cl \mathcal D(I_X) \neq \cl C_X$  by \cite[Corollary~6.3.1]{Rockafellar}. Therefore, there exists an open ball $B \subset \R^d$ such that $\cl \mathcal D(I_X) \cap \cl B = \varnothing $ and $\cl C_X \cap B \neq \varnothing$. 

For any $n \in \N$, let $S_n$ be the sum of $n$ independent identically distributed copies of $X$. Then for any $u \in \mathcal D(K_X)$, we have
\[
\P(S_n/n \in B)    = \E [\I(S_n/n \in B)] \le  \E \big [ \I \big ( u \cdot S_n  \ge n \inf_{v \in \cl B} u \cdot v \big ) \big ] \le e^{-n\inf_{v \in \cl B} u \cdot v} \, \E e^{u \cdot S_n},
\]
where the last equality follows from Markov's inequality. Then
\[
n^{-1} \log \P(S_n/n \in B) \le \inf_{u \in \mathcal D(K_X)} \Big( - \inf_{v \in \cl B} (u \cdot v -K_X(u)  ) \Big)= - \sup_{u \in \mathcal D(K_X)} \inf_{v \in \cl B} (u \cdot v - K_X(u) ).  
\]
Finally, let us interchange the supremum and the infimum using a minimax result~\cite[Corollary~37.3.2]{Rockafellar} on concave-convex functions. This gives
\begin{equation} \label{eq: Cramer bound}
n^{-1} \log \P(S_n/n \in B) \le - \inf_{v \in \cl B} I_X(v).
\end{equation} 
This inequality appears, e.g., in~\cite[Eq.~(2.16)]{FirasTimo}. 

On the other hand, since $\cl C_X \cap B \neq \varnothing$, $B$ is open, and $C_X$ is convex, it follows  from \cite[Corollary~6.3.2]{Rockafellar} that $\rint C_X$ intersects with $B$. Hence, by Carath\'eodory's theorem (\cite[Theorem~17.1]{Rockafellar}), there is a convex combination $\sum_{i=1}^m \alpha_i x_i \in B$, where $m $ is a positive integer, $x_i \in \supp X $ and $\alpha_i >0$ for every $1 \le i \le m$, and $\sum_{i=1}^m \alpha_i=1$. By finding a rational approximation to all but one of the $\alpha_i$'s, we get $\frac1n \sum_{i=1}^m n_i x_i \in B$ for some positive integer $n_i$ and $n= \sum_{i=1}^m n_i$. Furthermore, there exist open balls $B_i \subset \R^d$ such that $x_i \in B_i$ for  every $1 \le i \le m$ and $\frac1n \sum_{i=1}^m n_i B_i \subset B$. Since each open ball $B_i$ intersects with $\supp  X$,  we have $\P(X \in B_i)>0$. Therefore, 
\begin{equation} \label{eq: lower bound}
\P(S_n/n \in B) \ge \Pi_{i=1}^m \P(X \in B_i)^{n_i}>0.
\end{equation}

Inequalities \eqref{eq: Cramer bound}  and \eqref{eq: lower bound} imply that $\cl B \cap \mathcal D(I_X) \neq \varnothing$, which is contradiction. Thus, we proved that $\rint C_X \subset \mathcal{D} (I_X)  \subset  \cl C_X$, establishing \eqref{eq: D_I inclusions}. Finally, by \cite[Corollary~6.3.1]{Rockafellar} this gives  $\rint C_X = \rint \mathcal{D} (I_X)$, as required.




\ref{item: I strictly convex}) Put $d':=\dim C_X$. We assume that $d' \ge 1$, otherwise the claim is trivial.

Recall that $I_X$ is {\it subdifferentiable} at a point $v_0 \in \R^d$ if  there is a $u \in \R^d$ such that the inequality $I_X(v) \ge I_X(v_0) + u \cdot (v - v_0)$ holds for every $v \in \R^d$. We claim that $I_X$ is subdifferentiable at no point outside of $\rint \mathcal D(I_X)$. Combined with the fact that $K_X$ is differentiable at every point of $\intr \mathcal D(K_X)$ (\cite[\mbox{Corollary}~7.1]{B-N}), this implies  that the asserted necessary and sufficient condition for strict convexity of $I_X$ is a particular case of \cite[Theorem~26.3]{Rockafellar}. 

Assume first that $d'=d$. In this case $K_X$ is strictly convex by \cite[Theorem~7.1]{B-N}; this actually follows immediately from the criterion for equality in H\"older's inequality. Therefore, $K_X$ is {\it essentially strictly convex}, i.e.\ $K_X$ is strictly convex on every interval contained in the set of points where $K_X$ is subdifferentiable. Hence $I_X$ is {\it essentially smooth} by \cite[Theorem~26.3]{Rockafellar}, that is $I_X$ is differentiable on the set $\intr \mathcal D(I_X)$, which is required to be non-empty, and $I_X$ is steep. By \cite[Theorem~26.1]{Rockafellar}, this implies that $I_X$ is not subdifferentiable outside of $\rint \mathcal D(I_X)$, as required.

In the remaining case $1 \le d' \le d-1$, put $L:=\aff(\supp X)$. We can assume w.l.o.g.\ that $0 \in L$, otherwise pick any $\mu \in L$ and use the simple fact that $I_X(v)=I_{X-\mu}(v-\mu)$ for $ v \in \R^d$, which easily implies that our claim holds true for $I_X$ if and only if it holds for $I_{X - \mu}$.

Since $L$ is a linear subspace of $\R^d$ of dimension $d'$, there exists an orthogonal mapping $U: L \to \R^{d'}$. Then by $X \in L$ a.s., for any $v \in L$ we have
\begin{equation} \label{eq: I mapped}
I_X(v)= \sup_{u \in \R^d} \bigl( u \cdot v - \log \E e^{u \cdot X} \bigr) = \sup_{u \in L} \bigl( u \cdot v - \log \E e^{u \cdot X} \bigr) = I_{U(X)} (U(v)),
\end{equation}
where in the last equality we used the change of variables $u \mapsto U(u)$.
Therefore, since the mapping $U$ is linear and invertible, $I_X$ is subdifferentiable at a $v\in L$ if and only if $I_{U(X)}$ is subdifferentiable at $U(v)$ by \cite[Theorem~23.9]{Rockafellar}  (applied with $f=I_{U(X)}$ and $A=U^{-1}$). On the other hand, $v \in \rint \mathcal D(I_X) $ if and only if $U(v) \in \rint \mathcal D_{U(X)}$, since $U(\rint \mathcal D(I_X) ) = \rint U(\mathcal D(I_X) )  = \rint \mathcal D(I_{U(X)})$ by  \cite[Theorem~6.6]{Rockafellar}. Thus, by equality \eqref{eq: I mapped}, the case $d'<d$ reduces to the case $d'=d$ because the support of the random vector $U(X)$ in $\R^{d'}$ has full dimension. This finishes the proof of the claim.
\end{proof}

Our proofs of Theorems~\ref{thm: restriction} and~\ref{thm: main} rely on the following technical result, where $^*$ stands for convex conjugation (the Legendre--Fenchel transform) of functions on $\R^d$.
\begin{lemma} \label{lem: technical}
Let $X$ be a random vector in $\R^d$, $d \ge 1$, and $L$ be a hyperplane in $\R^d$ supporting $C_X$ and such that $\P(X \in L) >0$. Put $\tilde K_{X|L} (u):= K_{X|L}(u)$  if $ u \in \pr_L^{-1}(\pr_L \mathcal D(K_X))$, otherwise $\tilde K_{X|L} (u):=  +\infty$ for  $u \in \R^d$. Then
\begin{equation} \label{eq: I_X on L}
I_X(v)=(\tilde K_{X|L})^*(v) - \log \P(X \in L), \qquad v \in L,
\end{equation}
and $(\tilde K_{X|L})^*(v) = +\infty$ for $v \not \in L$. Moreover, we have
\begin{equation} \label{eq: technical proj}
\pr_L  (\rint \mathcal D(\tilde K_{X|L}) ) = \pr_L  (\rint \mathcal D(K_X) ).
\end{equation}
\end{lemma}

\begin{proof}
Denote by $L_0$ the hyperplane passing through $0$ that is parallel to $L$, and let $\ell \in \R^d$ be the unit vector orthogonal to $L_0$ such that $\ell \cdot u \le \ell \cdot v$ for any $u \in  C_X$ and $v \in L$. Denote by $(v_1, v_2)$ the coordinates of $v \in L$ in $L_0 \oplus L^\bot$, where $L^\bot:= \R \ell$. 

For any $u_1 \in L_0$ such that $\E e^{(u_1 + u_2 \ell ) \cdot X} < \infty$ for some real $u_2=u_2'$, we have
\begin{align} \label{eq: boundary point}
\sup_{u_2 \in \R }  \bigl ( u_2 v_2 - \log \E e^{(u_1 + u_2 \ell ) \cdot X} \bigr) &=
 - \log \bigl ( \inf_{u_2 \in \R }  \E e^{u_1 \cdot X + u_2 (\ell \cdot X - v_2)} \bigr) \notag \\
&=- \log \E [e^{u_1  \cdot  X} \I_{\{X \in L\}}],
\end{align}
where the last equality follows from the dominated convergence theorem applied as $u_2 \to +\infty$ using that $e^{u_1 \cdot X + u_2' (\ell \cdot X - v_2)}$ is an integrable majorant, which is true because the function $u_2 \mapsto e^{u_1 \cdot X + u_2 (\ell \cdot X - v_2)} $  is non-increasing a.s.\ by $\ell \cdot X \le v_2$~a.s. On the other hand, if $u_1 \in L_0$ is such that $\E e^{(u_1 + u_2 \ell ) \cdot X} = \infty$ for every real $u_2$, then the l.h.s.\ of the first equality in \eqref{eq: boundary point} is $-\infty$. Therefore, 
\begin{align*} 
I_X(v) &= \sup_{u_1 \in L_0} \sup_{u_2 \in \R }  \bigl (u_1 \cdot v_1 + u_2 v_2 - \log \E e^{(u_1 + u_2 \ell ) \cdot X} \bigr)\\
&= \sup_{\substack{u_1 \in L_0:\\ (u_1 + \R \ell) \cap \mathcal D(K_X) \neq \varnothing}} \bigl (u_1 \cdot v_1 - \log \E [e^{ u_1 \cdot  X} \I_{\{X \in L\}}]\bigr) \\
&= \sup_{u_1 \in \pr_{L_0} (\mathcal D(K_X)) } \sup_{u_2 \in \R} \bigl (u_1 \cdot v_1 + u_2 v_2 - \log \E \big [e^{ (u_1 + u_2 \ell ) \cdot  (X|L)} \big]\bigr) - \log \P(X \in L) \\
&= \sup_{u\in \pr_L^{-1}(\pr_L \mathcal D(K_X))} \bigl (u \cdot v - \log \E e^{ u \cdot  (X|L)} \bigr) - \log \P(X \in L), 
\end{align*}
where in the last equality we used that  $L_0$ and $L$ are parallel. This proves~\eqref{eq: I_X on L}.

The claim $(\tilde K_{X|L})^*(v)=+\infty$ for $v \not \in L$ follows exactly as in~\eqref{eq: I = infty} using that $\tilde K_{X|L}(u_0)=K_{X|L}(u_0)<+\infty$ for any $u_0 \in L^\bot$ by $0 \in \mathcal D(K_X)$. 

Lastly, it follows from \eqref{eq: K proj ineq} that $K_{X|L}(u) < + \infty$ when $u \in \pr_L^{-1}(\pr_L \mathcal D(K_X))$. Therefore, by the definition of $\tilde K_{X|L}$, we have $\mathcal D( \tilde K_{X|L}) = \pr_L^{-1}(\pr_L \mathcal D(K_X))$, and we obtain \eqref{eq: technical proj}  interchanging $\pr_L$ and $\rint$ by \cite[Theorem~6.6]{Rockafellar} as follows:
\[
\pr_L(\rint \mathcal D( \tilde K_{X|L})) = \rint (\pr_L \mathcal D( \tilde K_{X|L})) = \rint (\pr_L \mathcal D(  K_X)) = \pr_L (\rint \mathcal D( K_X)).
\]
\end{proof}

\begin{proof}[{\bf Proof of Theorem~\ref{thm: restriction}}]
Let $L$ be a hyperplane supporting $C_X$ and such that $\P(X \in L)>0$. By Lemma~\ref{lem: technical} and Proposition~\ref{prop: known}.\ref{item: D_I inclusions} applied to $X|L$,  we have $I_{X|L}(v)= (K_{X|L})^*(v) = + \infty $ and $(\tilde K_{X|L})^*(v)=+\infty$ for $v \not \in L$. Therefore, the functions $(\tilde K_{X|L})^* $ and $  (K_{X|L})^*$ coincide if they are equal on $L$. Thus, \eqref{eq: I_X on L} implies that 
\[
I_X(v)=I_{X|L}(v) - \log \P(X \in L), \qquad v \in L,
\]
if and only if $(\tilde K_{X|L})^* =  (K_{X|L})^*$. This is in turn equivalent to $(\tilde K_{X|L})^{**} =  K_{X|L}$  (by \cite[Theorem~12.2]{Rockafellar}) because $\tilde K_{X|L}$ is a convex function (this follows from the definition of  $\tilde K_{X|L}$)  and $K_{X|L}$ is a lower semi-continuous convex function (by \cite[Theorem~7.1]{B-N}), both finite at least at one point. The last equality holds true if and only if $\tilde K_{X|L}$ equals $K_{X|L}$ except  possibly at some relative boundary points of $\mathcal D( \tilde K_{X|L})$ (by \cite[Theorem~7.4]{Rockafellar}). Thus, equalities~\eqref{eq: I_X = I_X|L} and $\rint \mathcal D(\tilde K_{X|L}) = \rint \mathcal D(K_{X|L})$ are equivalent.

The latter one is equivalent to $\pr_L (\rint \mathcal D(\tilde K_{X|L})) = \pr_L( \rint \mathcal D(K_{X|L}))$ because the sets $ \rint \mathcal D(\tilde K_{X|L}) $ and $\rint \mathcal D(K_{X|L})$ are right cylinders  by \cite[Corollary~6.6.2]{Rockafellar}. Hence, by~\eqref{eq: technical proj}, equalities \eqref{eq: I_X = I_X|L} and \eqref{eq: consistent proj} are equivalent, as claimed.
 \end{proof}

\begin{proof}[{\bf Proof of Theorem~\ref{thm: main}.}]
\ref{item: I on faces}) Let us prove the first inclusion in  \eqref{eq: D I_X max faces}. Let $F \in \mathcal F^*(\mathcal D(I_X))$ be a maximal proper face of $\mathcal D(I_X)$. Then there is a hyperplane $L$ supporting the convex set $\mathcal D(I_X)$ such that $F= \mathcal D(I_X) \cap L$.  The hyperplane $L$ also supports $\cl C_X$ by the second inclusion in \eqref{eq: D_I inclusions}. Moreover, we have $\P(X \in L)>0$ since otherwise $\mathcal D(I_X) \cap L = \varnothing$ by Proposition~\ref{prop: known}.\ref{item: D_I inclusions}, which is a contradiction. Therefore, $L \cap C_X \neq \varnothing$, and thus $L$ supports $C_X$. Hence $C:=C_X \cap L$ is a face of $C_X$. We also have $F=\mathcal D(I_{X|C})$ by  \eqref{eq: I_X = I_X|L} and the fact that $X|C$ has the same distribution as $X|L$ (as $\P(X \in L \setminus C) =0$). Hence $F= \mathcal D(I_X) \cap \cl C_{X|C}$~by~\eqref{eq: D_I inclusions}. 

Clearly, $C$ is a proper face of $C_X$ (i.e.\ $C \neq C_X$) since otherwise $F$ cannot be a proper face of $\mathcal D(I_X)$. However, $C$ is not necessarily a maximal proper face. Let $C' \in \mathcal F^*(C_X)$ be such that $C \subset C'$.  Since this is an exposed face of $C_X$, there is a hyperplane $L'$ supporting $C_X$ and satisfying $C'=C_X \cap L'$. We have 
\[
\P(X \in L')=\P(X \in C') \ge \P(X \in C)>0.
\]
Since $L'$ supports $C_X$, equality \eqref{eq: I_X = I_X|L}  is valid with $L=L'$ and it implies that the set $F':=\mathcal D(I_X) \cap L'$ satisfies $F'= \mathcal D(I_{X|C'})$ and therefore is non-empty; moreover, we have $F'= \mathcal D(I_X) \cap \cl C_{X|C'}$ by \eqref{eq: D_I inclusions}. This shows that $F'$  is a proper face of  $\mathcal D(I_X) $ since $L'$ supports $\mathcal D(I_X) $ by \eqref{eq: D_I inclusions}. 

Finally, by $C \subset C'$, we have $C_{X|C} \subset C_{X|C'}$, and thus 
\[
F = \mathcal D(I_X) \cap \cl C_{X|C} \subset \mathcal D(I_X) \cap \cl C_{X|C'} = F'.
\]
Therefore, $F=F'$ since $F$ is a maximal proper face by the assumption. 
Thus, we have $F= \mathcal D(I_{X|C'})$, which proves the first inclusion in   \eqref{eq: D I_X max faces}.

To  prove the remaining inclusion in   \eqref{eq: D I_X max faces}, pick a $C' \in \mathcal F^*_+(C_X)$. Then $C'=C_X \cap L'$ for some hyperplane $L'$ supporting $C_X$ and satisfying $\P(X \in L')>0$. As we have shown just above, $F':=\mathcal D(I_X) \cap L'$ is a non-empty proper face of $\mathcal D(I_X)$ (but it is not necessarily a maximal one anymore) and $F'=\mathcal D(I_{X|C'})$. This finishes the proof of Part~\ref{item: I on faces}). 

\ref{item: I strictly convex iff}) \underline{Direct implication.} Assume that $I_X$ is strictly convex. Let us use induction in $\dim C_X$ to prove that $\intr \mathcal D(K_X) \neq \varnothing$, $K_X$ has the projection property, and $K_X$ is totally steep.

This claim holds trivially in the base case $\dim C_X=0$, where $X$ is constant a.s.

To prove the induction step, consider any hyperplane $L$ supporting $C_X$ such that $0<\P(X \in L)<1$. By Lemma~\ref{lem: technical}, the effective domain of the function  $(\tilde K_{X|L})^*$ is contained in $L$. Therefore, this function is strictly convex by \eqref{eq: I_X on L} because $I_X$ is strictly convex by the assumption. Then $\intr \mathcal D((\tilde K_{X|L})^{**}) \neq \varnothing$  and  $(\tilde K_{X|L})^{**}$ is steep by~\cite[Theorem~26.3]{Rockafellar} (because strict convexity implies essential strict convexity). Hence $\intr \mathcal D(\tilde K_{X|L}) \neq \varnothing$  and $ \tilde K_{X|L}$ is also steep because it equals $(\tilde K_{X|L})^{**}$ except  possibly at some relative boundary points of $\mathcal D( \tilde K_{X|L})$ (\cite[Theorems~7.4 and~12.2]{Rockafellar}). 

Let us show that $\intr \mathcal D( \tilde K_{X|L}) = \intr \mathcal D(  K_{X|L})$. Otherwise, by $\mathcal D( \tilde K_{X|L})  \subset \mathcal D(  K_{X|L})$ and convexity of $ \mathcal D(  K_{X|L})$, there is a point $u \in \partial  \mathcal D( \tilde K_{X|L}) \cap \intr \mathcal D(  K_{X|L})$. Pick a sequence  $u_1, u_2, \ldots $ in $\intr \mathcal D( \tilde K_{X|L})$ converging to $u$. Then 
$\lim_{n\to \infty}|\nabla \tilde K_{X|L}(u_n)| = \infty$ by steepness of $\tilde K_{X|L}$. Thus, $\lim_{n\to \infty}|\nabla K_{X|L}(u_n)| = \infty$ because $ \tilde K_{X|L}$ equals $K_{X|L}$ whenever  $ \tilde K_{X|L} <\infty$, and hence $ \tilde K_{X|L} = K_{X|L}$ on $\intr \mathcal D( \tilde K_{X|L})$. However, it must be $\lim_{n\to \infty}|\nabla K_{X|L}(u_n)| = |\nabla K_{X|L}(u)|$ because $K_{X|L}$  is continuously differentiable on $\intr \mathcal D(  K_{X|L})$ since so is the Laplace transform of any random variable (\cite[\mbox{Corollary}~7.1]{B-N}). This is a contradiction.

We now have $\pr_L(\intr \mathcal D( \tilde K_{X|L})) = \pr_L(\intr \mathcal D(  K_{X|L})) $, which implies equality~\eqref{eq: consistent proj} by~\eqref{eq: technical proj}. Thus, $K_X$ satisfies Condition~\ref{item: consistent proj}) in Definition~\ref{def: proj property} of the projection property because $L$ was chosen arbitrarily. Equality~\eqref{eq: consistent proj} in turn implies \eqref{eq: I_X = I_X|L} by Theorem~\ref{thm: restriction}, hence $I_{X|L}$ is strictly convex because so is $I_X$ and $\mathcal D (I_{X|L}) \subset L$ by Proposition~\ref{prop: known}.\ref{item: D_I inclusions}.

For any maximal proper face $C \in \mathcal F^*_+(C_X)$, pick a hyperplane $L$ supporting $C_X$ such that $C= C_X \cap L$. Since $0<\P(X \in L)<1$ and $\dim(\supp (X|L)) < \dim(\supp X)$, we can apply the assumption of induction to the random vector $X|L$, which is distributed  as $X|C$ and has strictly convex rate function $I_{X|L}$ by the above. Therefore, $\intr \mathcal D(K_{X|C}) \neq \varnothing$, $K_{X|C}$ has the projection property, and $K_{X|C}$ is totally steep. Thus, since $C$ was chosen arbitrarily, $K_X$ has the projection property, as required. Finally, $K_X$ is totally steep, as required, since $\intr \mathcal D(K_X) \neq \varnothing$ and $K_X$ is steep by Proposition~\ref{prop: known}.\ref{item: I strictly convex}. 

\underline{Reverse implication.} Assume that $\varnothing \neq \intr \mathcal D (K_X)$, $K_X$ has the projection property, and $K_X$ is totally steep. We again use induction in $\dim C_X$ to show that $I_X$ is strictly convex.

In the base case $\dim C_X=0$, the set  $\mathcal D(I_X)$ consists of a single point, and the claim holds vacuously.

To prove the induction step,  pick a closed line segment $J \subset \mathcal D(I_X)$. From \eqref{eq: C without rint} and the definition of a face,  either $\rint J \subset \rint \mathcal D(I_X)$ or $J$ is contained in some maximal proper face of $\mathcal D(I_X)$. In the former case,  $I_X$ is not affine on $ J$ by Proposition~\ref{prop: known}.\ref{item: I strictly convex}. In the latter case, by the first inclusion in \eqref{eq: D I_X max faces}, we have $J \subset \mathcal D(I_{X|C}) \subset C$ for some face $C \in \mathcal F^*_+(C_X)$. Note that $\varnothing \neq \intr \mathcal D (K_X) \subset \intr \mathcal D(K_{X|C})$ (cf.~\eqref{eq: K proj ineq}), $K_{X|C}$ is totally steep since so is $K_X$, and $K_{X|C}$ has the projection property since $K_X$ has this property. Therefore, $I_{X|C}$ is strictly convex by $\dim(\supp (X|C))< \dim (\supp X)$ and the assumption of induction. Hence $I_X$ is not affine on $J$ because for some hyperplane $L$ supporting $C_X$ and such that $C=C_X \cap L$, we have $I_X=I_{X|L}=I_{X|C}$ by  \eqref{eq: I_X = I_X|L}, which holds true by Theorem~\ref{thm: restriction} because $K_X$ has the projection property. Therefore, $I_X$ is not affine on $J$  in either case and thus $I_X$ is strictly convex.
\end{proof}

\section*{Acknowledgements}
I thank the anonymous referee for comments and suggestions. This work was supported in part by Dr Perry James (Jim) Browne Research Centre.

\bibliographystyle{plain}

\begin{thebibliography}{10}

\bibitem{AkopyanVysotsky} A. Akopyan and V. Vysotsky. Large deviations of convex hulls of planar random walks and Brownian motions. Accepted in Ann. H. Lebesgue (2021) Available at arXiv:1606.07141.

\bibitem{B-N} O. Barndorff-Nielsen. Information and exponential families in statistical theory. Reprint of the 1978 original. John Wiley \& Sons, Ltd., Chichester, 2014.

\bibitem{BahadurZabell} R.R. Bahadur and S.L. Zabell. Large deviations of the sample mean in general vector spaces. Ann. Probab. (1979) {\bf 7}, 587--621.

\bibitem{DemboZeitouni} A. Dembo and O. Zeitouni. Large Deviations Techniques and Applications. Corrected reprint of the second edition. Springer-Verlag, Berlin, 2010.

\bibitem{Lifshits} M.A. Lifshits. Large deviation principle: processes and empirical distributions (in Russian). St.\ Petersburg State University, St.\ Petersburg, 2002. \url{https://sites.google.com/site/mlprobability/home10/ml09/Mall4.pdf}

\bibitem{Mogul} A.A. Mogulskii. Large deviations for processes with independent increments. Ann. Probab. (1993) {\bf 21}, 202--215.

\bibitem{FirasTimo} F. Rassoul-Agha and T. Sepp\"{a}l\"{a}inen. A Course on Large Deviations with an Introduction to Gibbs Measures. American Mathematical Society, Providence, RI, 2015.

\bibitem{Rockafellar} R.T. Rockafellar. Convex Analysis. Princeton University Press, Princeton, NJ, 1970.

\bibitem{Widder} D.V. Widder. The Laplace Transform. Princeton University Press, Princeton, NJ, 1941.

\end{thebibliography}

\end{document}